\documentclass[a4paper,10pt]{amsproc}

\usepackage{amsfonts, amsmath, amssymb, graphicx, mathrsfs}
\usepackage[all]{xy}

\newdir{ >}{!/6pt/@{ }*:(1,0.2)@_{>}*:(1,-0.2)@^{>}}

\theoremstyle{plain}

\newtheorem{theorem}{Theorem}[section]
\newtheorem{lemma}[theorem]{Lemma}
\newtheorem{proposition}[theorem]{Proposition}
\theoremstyle{definition}
\newtheorem{definition}[theorem]{Definition}

\newtheorem{counter example}[theorem]{Counter Example}

\newtheorem{corollary}[theorem]{Corollary}
\newtheorem{example}[theorem]{Example}

\numberwithin{equation}{section}

\author[S. Bag]{Sagarmoy Bag}
\address{Department of Pure Mathematics, University of Calcutta, 35, Ballygunge Circular Road, Kolkata 700019, West Bengal, India}
\email{sagarmoy.bag01@gmail.com}
\author[R.C.Manna]{Ram Chandra Manna}
\address{Ramakrishna Mission Vidyamandir, Belur Math, Howrah-711202, West Bengal, India}
\email{mannaramchandra8@gmail.com}
\author[S. K. Patra]{Sourav Kanti Patra}
\address{Ramakrishna Mission Vidyamandir, Belur Math, Howrah-711202, West Bengal, India}
\email{souravkantipatra@gmail.com}


\begin{document}

\title[Almost separable spaces]{Almost separable spaces}



\thanks {}
\maketitle		
\section*{Abstract}
 We have defined almost separable space. We show that like separability, almost separability is $c$ productive and converse also true under some restrictions. We establish a Baire Category theorem like result in Hausdorff, Pseudocompacts spaces. We investigate few relationships among separability, almost separability, sequential separability, strongly sequential separability.

\section{Introduction}
Let $X$ be any topological space. $C(X)$ be the set of all real valued continuous functions on $X$. A subset $A$ of $X$ is called almost dense in $X$ if for any $f\in C(X)$ with $f(A)=\{0\}$, implies that $f(X)=\{0\}$. And a topological space $X$ is called almost separable if it has a countable almost dense subset. A dense
subset is always almost dense. In completely regular space, dense and almost dense sets are identical. But the converse is not true. We give an example of a non completely regular space in which dense and almost dense sets are same [ Example $\ref{0}$ ]. Theorem $\ref{7}$ shows that almost separability is $c$ productive and under some restrictions the converse is also true Theorem $\ref{8}$. In theorem 4.3, we established Baire Category like theorem. Finally we establish relationships among almost separability, sequentially separability and strongly sequentially separability.
\begin{definition}\cite{B2013}
 A space $X$ is called sequentially separable if there exist a countable set $D$ such that for every $ x\in X$ there exist a sequence from $D$ converging to $x$
\end{definition}

\begin{definition}\cite{B2013}
A space $X$ is called strongly sequentially separable if it is separable and every dense countable subspace is sequentially dense. A subset $D$ of $X$ is called sequentially dense if for every $x$ in $X$ there exists a sequence from $D$ converges to $x$.
\end{definition}

\section{Almost dense subsets}

\begin{definition}
A subset $A$ of a topological space $X$ is called almost dense if for any $f\in C(X)$ with $f(A)=\{0\}$, implies $f(X)=\{0\}$.
\end{definition}

\begin{theorem}
Dense subsets are always almost dense.
\end{theorem}

\begin{proof}
Obvious.
\end{proof}

\begin{theorem}
Let $X$ be a completely regular space, then almost dense set are dense in $X$.
\end{theorem}

\begin{proof}
Let $A$ be almost dense subset of $X$. If possible let $A$ be not dense in $X$ and $x_\circ\in X\setminus \overline{A}$ then there exists $f\in C(X)$ such that $f(x_\circ )=0$ and $f(\overline{A})=1$. Then $f(A)=0$ but $f(X)\neq \{0\}$. This implies that $A$ is not almost dense in $X$ which is a contradiction. Hence $A$ is dense in $X$.
\end{proof}

We now show that the converse of the above theorem is not true.

\begin{theorem}\label{1}
Let $Y$ be an open dense subset of a topological space $X$ (may or may not be a Completely regular space). If $Y$ has a base of cozero subsets of $X$. Then almost dense set in $X$ is dense in $X$. A set $A$ of $X$ is called cozero set if $A=X\setminus Z(f)$, for some $f\in C(X)$, where $Z(f)=\{x\in X: f(x)=0\}$.
\end{theorem}

\begin{proof}
Let $A$ be almost dense subset of $X$. To show $A$ is dense in $X$. Suppose $A\cap Y=\emptyset$. Let $y\in Y$. Then there exist $f\in C(X)$ such that $y\in X\setminus Z(f)\subseteq Y$. Then $f(A)=\{0\}$ but $f(X)\neq \{0\}$, which is a contradiction as $A$ is almost dense in $X$. So, $A\cap Y\neq \phi$. Claim: $A\cap Y$ is dense in $Y$. If not there is a cozero set $X\setminus Z(g)\subset Y$ such that $(A\cap Y)\cap (X\setminus Z(g))=\empty$, where $g\in C(X)$. Then $Z(g)\subset (X\setminus Y)\cup (A\cap Y)\supset A$. So, that $g(A)=\{0\}$. But $g(X)\neq \{0\}$. Contradiction shows that $A\cap Y$ is d4ense in $Y$. Given that $Y$ is dense in $X$, $A\cap Y$ is dense in $Y$ which yields $A$ is dense in $X$.
\end{proof}

To find such a spaces we present following examples.

\begin{example}\label{0}
Let $K=\{\frac{1}{n}: n\in \mathbb{N}\}$ and $\beta =\{(a,b): a<b, a,b\in \mathbb{R}\}\cup \{ (a,b)\setminus K: a<b, a,b\in \mathbb{R}\}$ be a base for the topology $\tau_{K}$ on $\mathbb{R}$ which is known as $K$-topology on $\mathbb{R}$. Then $(\mathbb{R},\tau_K)$ is not regular and hence not completely regular. We would like to show that almost dense set in $(\mathbb{R},\tau_K)$ is dense in $(\mathbb{R},\tau_K)$. Let $Y=\mathbb{R}\setminus \{0\}$. The set $Y$ is open, dense in $(\mathbb{R},\tau_K)$ and the relative topology of $Y$ is usual topology of $Y$ which is completely regular and , as a result, possess a base consisting of cozero sets of $(\mathbb{R},\tau_K)$. We invoke theorem $\ref{1}$ to conclude that almost dense set in $(\mathbb{R},\tau_K)$ are dense in $(\mathbb{R},\tau_K)$.
\end{example}

\begin{example}
Let $X$ be zero-dimensional, non normal space. Let $F, G$ be two disjoint closed sets in $X$ which can not be separated by disjoint open sets. Then define $Y$ to be the quotient space obtaining by collapsing $G$ to a single point.
\end{example}

The following example shows that in a normal space the almost dense set and dense set may not be same.

\begin{example}
 Let $X=\{a,b\}$ and $\tau =\{\phi , X, \{a\}\}$ be a topology on $X$. Then $X$ is a normal space. $\{b\}$ is almost dense set in $X$ but not dense in $X$.
\end{example}

\begin{theorem}
Let $f:X\mapsto Y$ be a continuous onto function and $A$ be almost dense set in $X$, then $f(A)$ is almost dense set in $Y$.
\end{theorem}

\begin{proof}
Let $g\in C(Y)$ such that $g(f(A))=\{0\}$. Then $g\circ f\in C(X)$ such that $g\circ f (A)=\{0\}$. As $A$ is almost dense in $X$, then we have $g\circ f(X)=\{0\}$. Since $f$ is onto we have $g(Y)=\{0\}$. Hence $f(A)$ is almost dense in $Y$.
\end{proof}

\begin{theorem}
Let $\tau_1, \tau_2$ be two topologies on $X$ such that $\tau_2$ is finer than $\tau_1$. Then if a subset $A$ of $X$ is almost dense in $(X,\tau_2)$, then $A$ is almost dense in $(X,\tau_1)$.
\end{theorem}

\begin{proof}
We denote $C(X,\tau_i)$ for the set of all real valued continuous functions on $X$ with respect to the topology $\tau_i$ for $i=1,2$. Let $A$ be almost dense in $(X,\tau_2)$. Let $f\in C(X,\tau_1)$ with $f(A)=\{0\}$. Since $f\in C(X,\tau_2)$ as $\tau_2$ is finer than $\tau_1$. Since $A$ is almost dense in $(X,\tau_2)$, then $f(X)=\{0\}$. Hence $A$ is almost dense in $(X,\tau_1)$.
\end{proof}

\begin{theorem}
	If $A$ is almost dense in $X$ and $B$ is almost dense in $Y$, then $A\times B$ is almost dense in $X\times Y$.
\end{theorem}

\begin{proof}
Let $f:X\times Y\mapsto \mathbb{R}$ be continuous such that $f=0$ on $A\times B$. Fix $a\in A$ and look at $f_a(y)=f(a,y), y\in Y$. $f_a: Y\mapsto \mathbb{R}$ is continuous and $f_a(y)=0$ for all $y\in B$. By hypothesis $f_a(y)=0$ for all $y\in Y$.

Since $a$ is arbitrary we obtain $f(A\times Y)=\{0\}$. Take any $(x,y)\in X\times Y$. then $f_y: X\mapsto \mathbb{R}$ defined by $f_y(z)=f(z,y), z\in X$ is a continuous map. Now $f_y(a)=0$ for all $a\in A$. Then $f_y(x)=0$ for all $x\in X$ i.e. $f(x,y)=0$. Since $y\in Y$ is arbitrary this shows $f=0$ on $X\times Y$. Hence $A\times B$ is almost dense in $X\times Y$.
\end{proof}	

\begin{corollary}\label{4}
Let $A_1, A_2,...,A_n$ be almost dense in $X_1,X_2,....,X_n$ respectively. Then $\prod_{i=1}^{n}A_i$ is almost dense in $\prod_{i=1}^{n}X_i$.
\end{corollary}

\begin{proposition}
Let $\{X_\alpha :\alpha\in \Lambda\}$ be a family of topological spaces and $A_\alpha\subset X_\alpha, \alpha\in \Lambda$ be almost dense. Then $\prod_{\alpha\in\Lambda}A_\alpha$ is almost dense in $\prod_{\alpha\in \Lambda}X_\alpha$.
\end{proposition}

\begin{proof}
For $a=(a_\alpha)_{\alpha\in\Lambda }\in \prod_{\alpha\in\Lambda}A_\alpha$. Define $D=\{(x_\alpha)_{\alpha\in\Lambda}\in \prod_{\alpha\in\Lambda}X_\alpha:\{\alpha :x_\alpha\neq a_\alpha\}$ is finite$\}$. Let $f:\prod_{\alpha\in\Lambda}X_\alpha\mapsto \mathbb{R}$ be continuous such that $f=0$ in $\prod_{\alpha\in \Lambda}A_\alpha$. For any finite subset $I\subset \Lambda, \prod_{\alpha\in I}X_\alpha\times \{a_\alpha:\alpha\in\Lambda\setminus I\}$ is homeomorphic to $\prod_{\alpha\in I}X_\alpha$ and $\prod_{\alpha\in I}A_\alpha$ is almost dense in $\prod_{\alpha\in I}X_\alpha$(by above Corollary). Hence $\prod_{\alpha\in I}A_\alpha\times\{a_\alpha:\alpha\in\Lambda\setminus I\}$ is almost dense in $\prod_{\alpha\in I}X_\alpha\times\{a_\alpha:\alpha\in\Lambda\setminus I\}$. Consequently $f(\prod_{\alpha\in\Lambda}X_\alpha\times\{a_\alpha:\alpha\in\Lambda\setminus I\})=\{0\}$. Now $D=\cup[\prod_{\alpha\in I}X_\alpha\times\{a_\alpha:\alpha\in\Lambda\setminus I\}:I\subset \Lambda$ finite $]$. This implies that $f(D)=\{0\}$. We now invoke the following important result in product space.
\end{proof}

 We now invoke the following important result on product space.
\begin{theorem}
Let $\{Y_\alpha:\alpha\in \Lambda\}$ be a family of topological spaces and and $a=(x_\alpha)_{\alpha\in\Lambda}$ be a fixed element of $\prod_{\alpha\in\Lambda}Y_\alpha$. The set $E=\{(y_\alpha)_{\alpha\in \Lambda}\in \prod_{\alpha\in\Lambda}Y_\alpha:\{\alpha\in\Lambda :y_\alpha\neq x_\alpha\}$ is finite$\}$ is dense in $\prod_{\alpha\in\Lambda}Y_\alpha$.
\end{theorem}
Because of this theorem $D$ is dense in $\prod_{\alpha\in\Lambda}X_\alpha$ and $f(D)=0$ which implies $f=0$ on $\prod_{\alpha\in \Lambda}X_\alpha$. As a result $\prod_{\alpha\in \Lambda}A_\alpha$ is almost dense in $\prod_{\alpha\in\Lambda}X_\alpha$.

\begin{theorem}\label2{}
If a topological space $X$ contains a connected almost dense subset, then $X$ is connected.
\end{theorem}

\begin{proof}
Let $A$ be connected almost dense subset of $X$. If possible let $X$ be disconnected. Then there exists a onto continuous function $f:X\mapsto \{0,1\}$. Let $Y=\{x\in X: f(x)=0\}$. Then $X\setminus Y=f^{-1}\{1\}$. Since $A$ is connected, then either $A\subseteq Y$ or $A\subseteq X\setminus Y$.

Case 1: If $A\subseteq Y$. Then $f(A)=\{0\}$ but $f(X)\neq \{0\}$, which contradicts the fact that $A$ is almost dense set in $X$.

Case 2: If $A\subseteq X\setminus Y$. Consider $g= 1-f$. Then $g(A)=\{0\}$ but $g(X)\neq \{0\}$, which is a contradiction as $A$ is almost dense set in $X$.\\

Hence $X$ is connected.
\end{proof}

Of course the converse of $\ref{2}$ is not true i.e., $X$ is a connected space but an almost dense subset need not be connected. As an illustration consider the following:

$X=\mathbb{R}$ and the topology consist of all subsets of $\mathbb{R}$ containing $0$. $X$ is obviously connected. The subspace $A=\{1,2\}$ is almost dense in $X$ but not connected.

\begin{theorem}\label{3}
A subset $A$ of $X$ is almost dense in $X$ if and only if every non empty cozero set intersects $A$.
\end{theorem}

\begin{proof}
Let $A$ be an almost dense subset of $X$ and $ X-Z(f) $ be a non-empty cozero set in $X$. Let us assume that $ A\cap ( X-Z(f))=\emptyset$. Then $ f(A)=\{0\}$. Since $A$ is almost dense in $X$ , then  $ f(X)= \{0\}$ ,Which contradicts the fact that $ X-Z(f)$ is non-empty. So $ A\cap\ (X-Z(f))\neq\emptyset$.

Conversely, let $A$ intersect every non-empty cozero sets of $X$. Suppose $A$ is not almost dense subset of $X$. Then there exists a continuous function $ f:X\mapsto\mathbb{R}$ such that $ f(A)=\{0\}$ but $ f(X)\neq\{0\}$. Then $X-Z(f)$ is a non empty cozero set and $ A $ does not intersect the non-empty cozero set $ X-Z(f)$.
\end{proof}

As is well known that in a topological space a non-empty subset is dense if and only if it intersect every non-empty open set. Theorem $\ref{3}$ presents the analogous result in the case of almost denseness property. We have seen that in a topological space $X$ if $A\subset X$ is dense and $U$ is non-empty open set, then not only $A\cap U\neq \emptyset$, $A\cap U$ is dense in $U$ also. Now the question arises as follows:

Let $A$ be almost dense in a topological space $X$ and $\emptyset\neq U\subset X$ is a cozero set. Then $A\cap U\neq \emptyset$ is no doubt, but $A\cap U$ is almost dense in $U$? The answer does not seem to be clear!

\section{Almost separable spaces}

\begin{definition}
A space $X$ is called almost separable if it contains a countable almost dense subset.
\end{definition}

\begin{theorem}
Each separable space is almost separable.
\end{theorem}

The converse of the above theorem is not true.

\begin{example}
Let $X=\mathbb{R}$ and $\tau_c$ be cocountable topology on $X$. For any countable set $A$ in $X$, $X\setminus A$ is open. Therefore $A$ is not dense in $X$. Therefore $X$ is not separable.

  We want to show that $\mathbb{Q}$ is almost dense in $X$. Let $f\in C(X)$ and $f (\mathbb{Q})=\{0\}$. Since each real valued continuous function on $X$ is constant, then $f(X)=\{0\}$. Therefore $\mathbb{Q}$ is almost dense in $X$. Therefore $X$ contains a countable almost dense subset. Hence $X$ is an almost separable space. In fact, any countable subset of $X$ is almost dense in $X$.
\end{example}

\begin{theorem}
Finite product of almost separable spaces is almost separable.
\end{theorem}

\begin{proof}
It follows from the Corollary $\ref{4}$.
\end{proof}

In the case of infinite products, like the case of separable spaces, the following result is true.

\begin{theorem}\label{7}
Let  $\{X_\alpha:\alpha\in \Lambda\}$ be a family of almost separable topological spaces, $\Lambda$, with card($\Lambda)= c$, Then $\prod_{\alpha\in \Lambda}X_\alpha$  with product topology is almost separable space.
\end{theorem}

\begin{proof}
Let $ A_\alpha\subset X_\alpha$ be a countable almost dense subset of $X_\alpha$.
To avoid triviality we assume $A_\alpha$ to be countably infinite. Let $f_{\alpha}:\mathbb{N}\mapsto A_{\alpha}$ be a bijection. Define $\prod_{\alpha \in \Lambda}f_\alpha:\mathbb{N}^{\Lambda}\mapsto \prod_{\alpha \in \Lambda}A_\alpha$ as follows: $\prod_{\alpha\in\Lambda}f_\alpha{(n_\alpha,\alpha\in\Lambda)}=(f_\alpha(n_\alpha),\alpha\in\Lambda)$
where $(n_\alpha,\alpha\in\Lambda)\in\mathbb{N}^{\Lambda}$. We write $f^{\Lambda}$ for $\prod_{\alpha\in\Lambda}f_{\alpha}$. We know $\mathbb{N}^{\Lambda}=\{g:\Lambda\mapsto\mathbb{N}\}$ and $ f^{\Lambda}(g)= (f_\alpha(g(\alpha)): \alpha\in \Lambda)\in \prod_{\alpha\in\Lambda}A_\alpha$. It is easy to see that $f^{\Lambda}$ is onto. Let $ p_{\alpha}:\prod_{\alpha\in\Lambda}A_{\alpha}\mapsto A_{\alpha}$ be the projection to $\alpha$-th coordinate. Then $p_{\alpha}\circ f^{\Lambda}(g)=f_\alpha\circ g(\alpha),\alpha\in \Lambda$.
This shows that $ f^{\Lambda}:\mathbb{N}^{\Lambda}\mapsto\prod_{\alpha\in\Lambda}A_{\alpha}$ is continuous and onto. It is well known that $\mathbb{N}^{\Lambda}$ is separable and, hence almost separable. Therefore $\prod_{\alpha\in\Lambda}A_{\alpha}$  is almost separable. Since $\prod_{\alpha\in\Lambda}A_{\alpha}$  is almost dense in $\prod_{\alpha\in\Lambda}X_{\alpha}$ ,
$\prod_{\alpha\in\Lambda}X_{\alpha}$ is almost separable.
\end{proof}

\begin{definition}
A topological space $X$ is called functionally Hausdorff if for any two distinct points $a, b\in X$ there exists $f\in C(X)$ such that $f(a)=0$ and $f(b)=1$.
\end{definition}

\begin{lemma}\label{5}
If $X$ a functionally Hausdorff space, then for any two distinct points $a, b\in X$ there exist two distinct cozero sets $C, D$ in $X$ such that $a\in C$ and $b\in D$.
\end{lemma}

\begin{proof}
$X$ is functionally Hausdorff. Then for $a, b\in X$ with $a\neq b$ there exists $f\in C(X)$ such that $f(a)=0$ and $f(b)=1$. Let $C=\{x\in X: f(x)<\frac{1}{2}\}=((f-\underline{\frac{1}{2}})\wedge \underline{0})^{-1}(\mathbb{R}\setminus\{0\})$ and $D=\{x\in X: f(x)>\frac{1}{2}\}=((f-\underline{\frac{1}{2}})\vee \underline{0})^{-1}(\mathbb{R}\setminus\{0\})$, where $(g\vee h)(x)=\max \{g(x), h(x)\}$ and  $(g\wedge h)(x)=\min \{g(x), h(x)\}$ for all $x\in X$. Thus $C, D$ are disjoint cozero sets in $X$ and $a\in C, b\in D$. This completes the proof.
\end{proof}

\begin{theorem}\label{8}
Let $X=\prod_{\alpha\in\Lambda}X_\alpha$ be almost separable space, where each $X_\alpha$ is functionally Hausdorff and contains at least two points. Then each $X_\alpha$ is almost separable and $card (\Lambda)\leq c$.
\end{theorem}

\begin{proof}
Let $D$ be a countable ,almost dense subset of $X$. Consider the projection function $ p_{\alpha}:X\mapsto X_{\alpha}$ to the $\alpha$-th coordinate. Then the function $ p_{\alpha}:X\mapsto X_{\alpha}$ is continuous. Since continuous image of an almost separable space is almost separable, therefore each $X_{\alpha}$ is almost separable.

For every $\alpha\in\Lambda$, let $ a_{\alpha}, b_{\alpha}\in X_{\alpha}$ with $ a_{\alpha}\neq b_{\alpha}$. Since each $ X_{\alpha}$ is functionally Hausdorff, therefore there exist disjoint cozero sets $C_\alpha, D_\alpha$ in $X_\alpha$ such that $a_\alpha\in C_\alpha$ and $b_\alpha\in D_\alpha$. $p_\alpha^{-1}(C_\alpha)$ is a non-empty cozero set in $X$. By Theorem $\ref3{}$ $K_\alpha=D\cap p_\alpha^{-1}(C_\alpha)\neq \emptyset$.

Define a function $\phi:\Lambda\mapsto\mathcal{P}(D)$ by $\phi(\alpha)=K_\alpha$. Now $p_\alpha^{-1}(C_\alpha)\cap p_\alpha^{-1}(D_\beta)$ is a non-empty cozero set in $X$. Then there exist $x\in D\cap p_\alpha^{-1}(C_\alpha)\cap p_\alpha^{-1}(D_\beta)$ by Theorem $\ref{3}$. So, $x\in K_\alpha$ and $x\notin K_\beta$. Therefore $K_\alpha\neq K_\beta$. Thus $\phi$ is injective. Hence $card (\Lambda)\leq card (\mathcal{P}(D))=c$. This completes the proof.
\end{proof}

\begin{theorem}\label{6}
For an almost separable space $X$, the cardinality of $C(X)$ is less than or equal to $c$.
\end{theorem}

\begin{proof}
Let $A$ be countable almost dense subset of $X$. Define a map $\phi :C(X)\mapsto C(A)$ by $\phi (f)=f\lvert_A$. We sow that $\phi$ is injective mapping. Let $\phi (f)=\phi (g)$, where $f,g\in C(X)$. Then $f\lvert_A=g\lvert_A$. Let $h=f-g$. Then $h\in C(X)$ and $h(A)=\{0\}$. Since $A$ is almost dense in $X$, then $h(X)=\{0\}$. Therefore $\phi$ is injective. Since the cardinality of $C(A)$ is less that or equal to $c$, therefore the cardinality of $C(X)$ is less than or equal to $c$.
\end{proof}

\begin{corollary}
If an almost separable space has a uncountable closed discrete subspace, then it is not normal.
\end{corollary}

\begin{theorem}
Let $X$ be functionally Hausdorff, almost separable space. Then the cardinality of $X$ is atmost $2^{c}$.
\end{theorem}

\begin{proof}
Consider the function $\psi:X\mapsto\mathcal{P}(C(X))$ by $\psi(x)= \{f\in C(X): f(x)= 0\}$. Use the functionally Hausdorff property of $X$ to conclude that $\psi$ is injective. By the Theorem $\ref{6}$, cardinality of $C(X)$ is $c$. Thus cardinality of $X$ is atmost $2^{c}$.

Alternative proof of the above theorem: Let $X$ be a functionally Hausdorff, almost separable space. Let $\tau_w$ be the weak topology on $X$ induced by $C(X)$. Then $(X,\tau_w)$ is completely regular Hausdorff space show that almost dense set becomes dense. Thus $(X,\tau_w)$ is a separable Hausdorff space. It is known that the cardinality of a separable Hausdorff space is atmost $2^{c}$.  	
\end{proof}

\begin{theorem}
Let $Y$ be almost dense in $X$ and $Y$ be almost separable as a subspace. Then $X$ is almost separable.
\end{theorem}

\begin{proof}
Let $A$ be countable almost dense subset of $Y$. Let $f\in C(X)$ such that $f(A)=\{0\}$. Then $f\lvert_Y\in C(Y)$ and $f\lvert_Y(A)=\{0\}$. This implies that $f\lvert_Y(Y)=\{0\}$ as $A$ is almost dense in $Y$. Since $Y$ is almost dense in $X$, then $f(X)=\{0\}$. Therefore $A$ is countable almost dense subset of $X$. Hence $X$ is almost separable space.
\end{proof}
\section{Baire Category Like Theorem}
\begin{theorem}
For a topological space $X$ the following are equivalent:\\$(i) X$ is pseudocompact\\$(ii)$ If  $\{F_{n}: n\in \mathbb{N}\}$ is a sequence of zero sets of $X$with finite intersection property , Then $\bigcap_{n=1}^{\infty}F_{n}\neq \emptyset$.\\$(iii)$ If $\{U_{n}: n\in \mathbb{N}\}$ is a countable cover of $X$ consisting of cozero sets , there exists a finite subcover.
\end{theorem}
\begin{theorem}
Let $X$ be a Hausdorff space. Given a non-empty cozero-set $U$ and $x\in U$ there exists a cozero-set $V$ and a zero-set $F$ such that $x\in V\subset F\subset U$.
\begin{proof}
Let $f:X\longrightarrow[0,1]$ such that $U=f^{-1}(0,1]=X-f^{-1}(\{0\}).$ Since $x\in U, f(x)>0$. Choose $\delta >0$ such that $0<f(x)-\delta <f(x)$. Let $V=f^{-1}(f(x)-\delta ,1]), F=f^{-1}[f(x)-\delta ,1]$. Then $x\in V\subset \subset F \subset U$. Now $(f(x)-\delta, 1]\subset [0,1] $ is an open set and hence a cozero-set and $[f(x)-\delta ,1]$ is a closed subset of $[0,1]$ and hence a zero-set. Then $V=f^{-1}(f(x)-\delta,1]$ is a cozero-set and $F=f^{-1}[f(x)-\delta,1]$ is a zero-set of $X$.	
\end{proof}	
\end{theorem}
\begin{theorem}({Baire Category Like Theorem})
Let $X$ be a Hausdorff, pseudocompact space. If $\{U_{n}:n\in \mathbb{N}\}$ is a sequence of almost dense cozero-sets of $X$, then $\bigcap _{n=1}^{\infty}U_{n}$ is a non-empty almost dense subset of $X$.
\begin{proof}
Write $D=\bigcap _{n=1}^{\infty}U_{n}$. To show $D\neq \emptyset $ and $D$ intersects every non-empty cozero-set. Let $V$ be a non-empty cozero-set and let $x\in V$. Since $U_{1}$ is almost dense , $V\cap U_{1}\neq \emptyset $ and is a cozero-set . Let $x_{1}\in V\cap U_{1}$. By theorem $(4.2)$, $ \exists $  cozero-set $V_{1}$ and a zero-set $F_{1}$ such that $x_{1}\in V_{1}\subset F_{1} \subset V\cap U_{1}$. Now $V_{1}\neq \emptyset $ cozero-set $\Longrightarrow  V_{1}\cap U_{2}\neq \emptyset $ and is a zero-set. Let $x_{2}\in V_{1}\cap U_{2}. \exists V_{2}$ a non empty co-zero set and $F_{2}$ , a zero-set , such that $x_{2}\in V_{2}\subset F_{2} \subset V_{1}\cap U_{2}\subset V\cap U_{1}\cap U_{2}$. Now $V_{2}\neq\emptyset$ cozero-set, $V_{2}\cap U_{3}\neq \emptyset$ and is a cozero -set. Let $x_{3}\in V_{2}\cap U_{3}$. $\exists V_{3}$, a cozero-set, and $F_{3}$ , a zero set , such that $x_{3}\in V_{3}\subset F_{3}\subset V_{2}\cap U_{3}\subset V\cap U_{1}\cap U_{2}\cap U_{3}$. Proceeding in this way we obtain non-empty cozero-set $V_{n+1}$, zero-set $F_{n+1}$ such that $x_{n+1}\in V_{n+1}\subset F_{n+1}\subset V_{n}\cap U_{n+1}\subset V\cap U_{1}\cap U_{2} \cap U_{3}\cap ......\cap U_{n+1}$ for all $n\geq 0$. Note that $F_{n+1}\subset F_{n}$ and $F_{n}'$ s are non-empty zero-sets. Since $X$ is pseudocompact , in virtue of the theorem $(4.1)$, $\bigcap _{n=1}^{\infty}F_{n}\neq \emptyset$.  		
Hence $\bigcap _{n=1}^{\infty}F_{n}\subset \bigcap _{n=1}^{\infty}(V\cap U_{1}\cap U_{2} \cap U_{3}\cap .......\cap U_{n})= V\bigcap_{n=1}^{\infty}U_{n}$ so that $V\cap (\bigcap_{n=1}^{\infty}U_{n})\neq \emptyset$. i.e, $V\cap D\neq \emptyset$. Since $V$ is an arbitary non-empty cozero-set , $D$ is almost dense.	
\end{proof}	
\end{theorem}		 	
We know that separability is not a hereditary property. Niemytzky's plane is a well known example. Same is true about almost separability as well. Niemytzky's plane provides an example in this case also.

We now finish our paper giving relations among the different types of separability notions which are defined earlier.

Strong sequentially separable $\Rightarrow$ Sequentially separable $\Rightarrow$ Separable $\Rightarrow$ Almost separable.

\textbf{Acknowledgments:} The authors would like to thank to Professor Alan Dow for Theorem 2.4. Also the authors wish to thank Professor Asit Baran Raha for their valuable suggestions that improved the article.

\bibliographystyle{plain}

\end{document}